\documentclass[11pt,a4paper]{amsart}

\usepackage{latexsym}

\usepackage[dvips]{graphics}

\usepackage{amssymb}
\usepackage{amsthm}
\usepackage{amsmath}
\usepackage{amstext}
\usepackage{amscd}
\usepackage{color}

\usepackage{ifmtarg}
\makeatletter
  \newcommand{\TODO}[1]{\@ifmtarg{#1}{\emph{\textcolor{red}{\textbf{TODO}}}~}{\textcolor{red}{ \emph{\textbf{TODO:}~#1~}}}}
  \newcommand{\FORUS}[1]{\@ifmtarg{#1}{\emph{\textcolor{blue}
  {\textbf{TODO}}}~}{\textcolor{blue}{ \emph{\textbf{FOR US:}~#1~}}}}
\makeatother

\numberwithin{equation}{section}

\theoremstyle{plain}
\newtheorem{thm}{Theorem}[section]
\newtheorem{lem}[thm]{Lemma}

\newtheorem{theorem*}{Theorem}[]

\theoremstyle{definition}
\newtheorem{defi}[thm]{Definition}

\newtheorem{rmk}[thm]{Remark}

\theoremstyle{definition}

\theoremstyle{remark}
\newtheorem{rem}[thm]{Remark}

\newtheorem{question}{Question}

\usepackage{enumitem}
\setlist[itemize]{labelindent=.6em, itemindent=1em, leftmargin=!, label=\textbullet}
\usepackage{ifmtarg}

\title{Regular projections and regular covers in O-minimal structures}

\author{M'hammed OUDRANE}
\address {Universit\'e C\^ote d'Azur,  CNRS,  LJAD, UMR 7351, 06108 Nice, France}
\email{oudrane@unice.fr}


\begin{document}

\begin{abstract}
 In this paper we prove that for any definable subset $X\subset \mathbb{R}^{n}$ in a polynomially bounded
 o-minimal structure, with $dim(X)<n$, there is a finite set of regular projections (in the sense of Mostowski ). We give also a weak version of this theorem in any o-minimal structure, and we give
   a counter example in o-minimal structures that are not polynomially bounded. As an application we show that in any o-minimal structure there exist a regular cover in the sense of Parusi\'nski.  
\end{abstract}
\maketitle

\begin{center}
\textbf{Introduction :}
\end{center}
The regular projection theorem was first stated by Mostowski (see \cite{MS}) for complex analytic hypersurfaces, and then later a sub-analytic version was proved by Parusi\'nski (see \cite{Pa1}). The theorem states that for any compact sub-analytic set $X\subset \mathbb{R}^{n}$ there exists a finite set of regular projections (see Definition~\ref{Def2.1} below). In \cite{Pa1}, Parusi\'nski showed that a generic choice of $n+1$ projections is sufficient. The theorem has many application, it leads to the proof of some important metric proprieties of sub-analytic sets, it has been used by Parusi\'nski in \cite{Pa3} to prove the existence of Lipschitz stratification of sub-analytic sets, more precisely to show that every compact sub-analytic set can be decomposed in a finite union of L-regular sets in such a way that it is easy to glue Lipschitz stratifications of the pieces. Different ways to prove the existence of L-regular decomposition in any o-minimal structures were given by Kurdyka \cite{K}, Kurdyka and Parusi\'nski \cite{KP}, and Paw\l{}ucki \cite{Paw}. The regular projection theorem has been also used recently by Parusi\'nski in \cite{Pa2} to prove the existence of regular cover for sub-analytic relatively compact open subsets which is used in \cite{GS} and \cite{L} to construct Sobolev sheaves on the sub-analytic site.  \\
In \cite{Pa2}, Parusi\'nski asked if the regular projection theorem and the existence of regular covers can be shown in any o-minimal structures. A first answer was given by Nguyen \cite{Nh}, he gives a Lipschitz version of this theorem, means that we can always find a finite set of regular projections after  applying a bilipschitz homeomorphism of the ambient space. Nguyen's proof is based on cell decomposition and Valette's result on the existence of a good direction (see \cite{V}). A direct consequence of this is the existence of regular covers in any o-minimal structure. \\
In this paper we give an answer to Parusi\'nski's question, we will prove that the regular projection theorem works in any polynomially bounded o-minimal structure, and we will give an example showing that this is no longer true in non-polynomially bounded o-minimal structures. We will also give a weak version of the regular projection theorem in any o-minimal structure, which is a straightforward consequence of  the techniques used in \cite{Nh}. We will use this weak version to adopt the proof in \cite{Pa2} to show  the existence of regular covers in any o-minimal structure. Our fundamental tools for the proof will be the cell decomposition and Miller's result (\cite{C1}) to find a replacement for the Puiseux with parameter argument used in \cite{Pa1}.             

\newpage

\section{Definitions, terminology, notation :}
\subsection{Notations:}
\begin{itemize}

\item $\mathcal{P}(X)$: the set of subsets of $X$.
\item $grad(f)$: is the gradient of a $C^{1}$ function $f$. 
\item if $f:A\times B \longrightarrow C$ is a map, then for $b\in B$ we denote by $f(.,b)$ (or also by $f_{b}$) the map,
\begin{center}
$f(.,b):A\longrightarrow C $\\
$\; \; \; \; x \mapsto f(x,b)$
\end{center}
 
\item $B(v,r)$ is the open ball of radius $r$ and center $v$, and $\overline{B}(v,r)$ is the closed ball of radius $r$ and center $v$. 
\item $Reg^{p}(X)$ : the set of points $x\in X$ such that $X$ is a $C^{p}$ manifold near $x$.
\item For $v\in \mathbb{R}^{n-1}$, $\pi _{v}:\mathbb{R}^{n}\longrightarrow \mathbb{R}^{n-1}$ is the linear projection parallel to $Vect((v,1))$.
\item For a set $A\subset \mathbb{R}^{n}\times \mathbb{R}^{m}$, for $x_{0} \in \mathbb{R}^{n}$, we denote by $A_{x_{0}}$ the set:
\begin{center}
 $A_{x_{0}}=\{y\in \mathbb{R}^{m} \; : \; (x_{0},y)\in A \;     \}$
 \end{center} 
\item If $F:\mathbb{R}^{n}\times \mathbb{R}^{m}\longrightarrow \mathbb{R}$ is $C^{1}$, we denote by $D_{1}F$ the differential of $F$ with respect to the first variable, or we denote it also by $D_{y}F$, and we denote by $D_{x}F$ or $D_{2}F$ the differential of $F$ with respect to the second variable.

\item $\overline{A}$ denotes the topological closure of $A$.

\item for a set $U \subset \mathbb{R}^{n}$, $\partial U$ denotes the frontier of $U$, i.e $\partial U=\overline{U}\setminus U$. 

\item $B\subset  \mathbb{R}^{k}$ is called an open (closed) box if it can written as a product:

\begin{center}
  $B=I_{1}\times ....\times I_{k}$, 
  \end{center}  
where the $I_{i}$ are open (closed) intervals in $\mathbb{R}$.

\item We denote by $\mathbb{N}$ the set of nonnegative integers.

\item If $A\subset \mathbb{R}$, we denote by $A_{+}$, $A_{-}$, and $A^{\ast}$ as the following:

\begin{center}
$A_{+}=\{x\in A \; : \; x\geqslant 0   \}$\\
$A_{-}=\{x\in A \; : \; x\leqslant 0   \}$\\
$A^{\ast}=\{x\in A \; : \; x\neq 0   \}$
\end{center}
And we denote by $A_{+}^{\ast}:=A_{+} \cap A^{\ast}$ and $A_{-}^{\ast}:=A_{-} \cap  A^{\ast}$.

\item a map $f: \mathbb{R} \longrightarrow \mathbb{R}$ is said to be ultimately equal to $0$ if there is $M\in \mathbb{R}$ such that $f(x)=0$ for all $x>M$. 

\end{itemize}

\subsection{O-minimal structures:}
An o-minimal structure on the field $(\mathbb{R},+,.)$ is a family $\mathcal{D}=(\mathcal{D}_{n})_{n\in \mathbb{N} }$ such that for any $n$, we have:
\begin{itemize}

\item[(1)] $\mathcal{D}_{n} \subset \mathcal{P}(\mathbb{R}^{n})$ is stable by complement and finite union.
\item[(2)] For any $P\in \mathbb{R}[X_{1},...,X_{n}]$, we have $Z(P) \in \mathcal{D}_{n}$, where $Z(P)=\{ x \in \mathbb{R}^{n}\; : \; P(x)=0 \}$.
\item[(3)] $\pi (\mathcal{D} _{n})\subset \mathcal{D}_{n-1}$, where $\pi : \mathbb{R}^{n} \longrightarrow \mathbb{R}^{n-1}$ is the standard projection.
\item[(4)] For any $A\in \mathcal{D}_{1} $, $A$ is a finite union of points and intervals.
\end{itemize}
For a fixed o-minimal structure $\mathcal{D}$, we have the definitions:
\begin{itemize}
\item[($\bullet$)] Elements of $\mathcal{D}_{n}$ are called definable sets.
\item[($\bullet$)] If $A\in \mathcal{D}_{n}$ and $B\in \mathcal{D}_{m}$, then a map $f:A\longrightarrow B$ is called a definable map if its graph is a definable set.
\item[($\bullet$)] The o-minimal structure $\mathcal{D}$ is said to be polynomialy bounded if for each definable function $f:\mathbb{R}\longrightarrow \mathbb{R}$, there is some $n\in \mathbb{N}$ such that $\mid f(x) \mid \leq x^{n}$ for $x$ big enough. 
\item[($\bullet$)] For $p\in \mathbb{N}\cup \{\infty  \}$ a definable $C^{p}$-manifold is a manifold with finite atlas of definable transition maps and definable domains, see e.g. \cite{F}. Most of the notions from the $C^{p}$ manifolds can be extended to definable $C^{p}$ manifolds in an obvious way.

\end{itemize}

\begin{thm} (Miller's Dichotomy \cite{C2}). 
An o-minimal structure $\mathcal{D}$ is either polynomially bounded or contains the graph of the exponential function.
\end{thm}

\subsection{Cell decomposition}
Take $p\in \mathbb{N}$, a definable set $C$ in $\mathbb{R}^{n}$ is said to be a $C^{p}$-cell if:

\begin{itemize}

\item[case $n=1$:] $C$ is either a point or an open interval.
\item[case $n\geq 2$:] $C$ is one of the following:

\begin{itemize}
\item[$\bullet$] $C=\Gamma _{\phi}$ (the graph of $\phi$), where $\phi :B\longrightarrow \mathbb{R}$ is a $C^{p}$ definable function, where $B$ is $C^{p}$-cell in $\mathbb{R}^{n-1}$. 
\item[$\bullet$] $C=\Gamma (\phi , \varphi)=\{ (x,y) \in B\times \mathbb{R} \; \; \phi (x)<y<\varphi (x)  \}$, where $\phi$ and $\varphi$ are two $C^{p}$ definable functions on a $C^{p}$-cell $B$, such that $\phi < \varphi$ with the possibility of $\phi=-\infty$ or $\varphi =+\infty$.

\end{itemize}

\end{itemize}

A $C^{p}$-cell decomposition of $\mathbb{R}^{n}$ is defined by induction as follows:

\begin{itemize}

\item[$\bullet$] A $C^{p}$-cell decomposition of $\mathbb{R}$ is finite partition by points and open intervals.
\item[$\bullet$] A $C^{p}$-cell decomposition of $\mathbb{R}^{n}$ is a finite partition $\mathcal{A}$ of $\mathbb{R}^{n}$ by $C^{p}$-cells, such that $\pi (\mathcal{A})$ is a $C^{p}$-cell decomposition of $\mathbb{R}^{n-1}$, where $\pi : \mathbb{R}^{n}\longrightarrow \mathbb{R}^{n-1}$ is the standard projection and $\pi (\mathcal{A})$ is the family:

\begin{center}
$\pi (\mathcal{A})=\{\pi (A) \; : \; A\in \mathcal{A}    \; \}$.
\end{center}
\end{itemize}

\begin{thm}
Let $p\in \mathbb{N}$ and $\{ X_{1},...,X_{n} \}$ be a finite family of definable sets of $\mathbb{R}^{n}$. Then there is a $C^{p}$-cell decomposition of $\mathbb{R}^{n}$ compatible with this family, i.e. each $X_{i}$ is a union of some cells.
\end{thm}

\begin{proof}
See \cite{M} or \cite{VD}.
\end{proof} 
Now we can define the dimension of a definable set.
Take $X$ a definable subset of $\mathbb{R}^{n}$ and $\mathcal{C}$ a cell decomposition of $\mathbb{R}^{n}$ compatible with $X$, then we define the dimension :

\begin{center}
$dim_{\mathcal{C}}(X)=max\{ dim(C) \; : \; C\subset X \; and \; C\in \mathcal{C} \}$.
\end{center}
This number does not depend on $\mathcal{C}$, we denote it by $dim(X)$.

\section{Regular projection theorem}
Let $X\subset \mathbb{R} ^{n}$ be a definable
 subset of $\mathbb{R}^{n}$. For any $\lambda \in \mathbb{R}^{n-1}$,  we  denote  by $\pi _{ \lambda}:\mathbb{R}^{n}\longrightarrow \mathbb{R}^{n-1}$ , the  projection  parallel  to  the  vector $(\lambda , 1) \in \mathbb{R}^{n}$. Fix $\varepsilon$ and $C$ positive real numbers and $p \in \mathbb{N}^{\ast}$. For $v\in \mathbb{R} ^{n-1}$ and $x\in \mathbb{R}^{n}$ we define the cone $C_{\varepsilon}(x,v)$ by:
 
 \begin{center}
  $C_{\varepsilon}(x,v)=\{ x+t(v',1) \; : \; t\in \mathbb{R}^{\ast} \; and \; v' \in B(v,\varepsilon)   \}$.
  \end{center}
\begin{defi}\label{Def2.1}
We say that the projection $\pi _{\lambda}$ is  $(\varepsilon,p)$-weak regular at a point $x\in \mathbb{R}^{n}$(with respect to $X$) if:

\begin{itemize}

\item[$(1)$] $(\pi _{\lambda}) _{\mid X}$ is finite.
 \item[$(2)$] the intersection of $X$ with the cone $C_{\varepsilon}(x,\lambda)$ is either empty or a disjoint finite union of sets of the form: 
\begin{center}
 $A_{f_{i}}=\{x+f_{i}(\lambda ')(\lambda ' , 1) \; : \; \lambda ' \in B(\lambda ,\varepsilon )  \}$,
 \end{center} 
where $f_{i}$ are non-vanishing $C^{p}$ functions defined on $B(\lambda , \varepsilon)$.
 \end{itemize}
 
 We say that the projection $\pi _{\lambda}$ is  $(\varepsilon,C,p)$-regular at a point $x\in \mathbb{R}^{n}$(with respect to $X$) if moreover  we have:
 
 \begin{itemize}
 
 \item[$(3)$] $\parallel grad(f_{i}) \parallel \leq C\mid f_{i} \mid$ on $B(\lambda , \varepsilon)$ for all $i$.
  
\end{itemize}

\end{defi}

\begin{rmk}
The notion of regular projection introduced by Mostowski in \cite{MS} is not exactly the same as the notion of good direction ( regular projection in the sense of Valette in \cite{V}, see also \cite{Nh} for comparing both notions ). One can show that a projection is regular in the sense of Mostowski if and only if it is week regular and good.   
\end{rmk}

\subsection{Weak regular projection theorem :}

Let $\mathcal{A}$ be an o-minimal structure on $\mathbb{R}$ (no condition on $\mathcal{A}$).

\begin{thm} \label{theorem2}
Let $X$ be a definable subset of $\mathbb{R}^{n}$ such that $dim(X) \leqslant n-1$, and $p\in \mathbb{N}^{\ast}$. Then there is $\varepsilon > 0$ and $\{v_{1},...,v_{k} \}\subset \mathbb{R}^{n-1}$ such that for every $x\in \mathbb{R}^{n}$ there is $i$ such that $\pi _{v_{i}}$ is $(\varepsilon,p)$-weak regular at $x$ with respect to $X$.

\end{thm}

\begin{proof} For the proof we need a few lemmas.

\begin{lem} \label{lemma1}  Take $C$ a definable subset of $\mathbb{R}^{n}$, and let $B$ to be a box in $\mathbb{R}^{k}$. Let $\Delta$ be a definable subset of $\mathbb{R}^{n}\times \mathbb{R}^{k}$ such that $dim(\Delta _{x})\leqslant (k-1)$, for all $x\in C$. Then there exists a finite definable partition $\mathcal{C}$ of $C$, such that for each $D\in \mathcal{C}$ there is a box $B_{D}\subset B$ such that we have:

\begin{center}
$(D\times B_{D})\cap \Delta =\emptyset$.
\end{center}
\end{lem}

\begin{proof}
See \cite{Nh}.
\end{proof}

\begin{rmk} We can easily extend Lemma~\ref{lemma1} to the case of $C^{p}$-definable manifolds ( we can replace $\mathbb{R}^{k}$ and $\mathbb{R}^{n}$ by definable manifolds ) by choosing charts and reducing to the case of the last lemma. 
\end{rmk}
\begin{lem} \label{lemma2}
Let $M$, $J$, and $N$ be definable manifolds, and let $f:M\times J\longrightarrow N$ be a $C^{1}$ definable submersion. Take $\mathcal{S}$ a finite collection of $C^{1}$ definable submanifolds of $N$, then:

\begin{center}
$\rho (f,\mathcal{S})= \{s\in J \; : \; f(.,s) \; is \; transverse \; to \; all \; elements \; of \; \mathcal{S} \;      \}$
\end{center}
is a definable set and we have : $dim(J\setminus \rho (f,\mathcal{S}))<dim(J)$.

\end{lem}

\begin{proof}
See [\cite{L3}, Lemma 3 ].
\end{proof}
Take $m\in \mathbb{N}$. Let $P_{X}(x,v,\varepsilon)$ be a property on $(x,v,\varepsilon) \in \mathbb{R}^{n}\times \mathbb{R}^{m} \times \mathbb{R}^{\ast}_{+}$, i.e a map $P:\mathbb{R}^{n}\times \mathbb{R}^{m} \times \mathbb{R}^{\ast}_{+} \longrightarrow \{0,1\}$. Denote by $\pi : \mathbb{R}^{n}\times \mathbb{R}^{m} \times \mathbb{R}\longrightarrow \mathbb{R}^{n}\times \mathbb{R}^{m}$ the standard projection.  

\begin{lem} If we have : 
\begin{itemize}
\item[(1)] for all $(x,v)\in \mathbb{R}^{n}\times \mathbb{R}^{m}$, $\varepsilon$ and $\varepsilon '$ such that $\varepsilon < \varepsilon '$,  we have:

\begin{center}
$P_{X}(x,v,\varepsilon ')$ is true $\Rightarrow$ $P_{X}(x,v,\varepsilon )$ is true. 
\end{center}
\item[(2)] $\mathcal{X}=\{ (x,v,\varepsilon) \in \mathbb{R}^{n}\times \mathbb{R}^{m} \times \mathbb{R}^{\ast}_{+} \; : \; P_{X}(x,v,\varepsilon) \; is \; true \;  \}$ is definable and for all $x\in \mathbb{R}^{n}$ $dim(\mathcal{B}_{x})<m$, where $\mathcal{B}$ is defined by:

\begin{center}
$\mathcal{B}=\{(x,v)\in \mathbb{R}^{n} \times \mathbb{R}^{m} \; ; \; \pi ^{-1}(x,v)\cap \mathcal{X}=\emptyset  \}$.
\end{center}
\end{itemize}
Then there is $\varepsilon _{0}>0$ and $\mathbb{A}=\{v_{1},...,v_{k} \}\subset \mathbb{R}^{m}$ such that for all $x\in \mathbb{R}^{n}$ there is some $i$ such that $P_{X}(x,v_{i},\varepsilon _{0})$ is true.
\end{lem}

\begin{proof}
By the assumptions $(1)$ and $(2)$, the function $F:\mathbb{R}^{n}\times \mathbb{R}^{m} \longrightarrow \overline{\mathbb{R}}_{+}$  given by:  
\begin{center}
$F(x,v)=\left\{
                        \begin{array}{ll}
                        sup \{\varepsilon : \; P_{X}(x,v,\varepsilon) \; is \; true      \} \; \; \; \; \; \;  if \{\varepsilon : \; P_{X}(x,v,\varepsilon) \; is \; true      \}\neq \emptyset \\
                        0 \; \; \; \; \; \; \; \; \; \; \; \; \; \; \; \; \; \; \; \; \; \; \; \; \; \; \; \; \; \; \; \; \; \; \; \; \; \; \; \; \; \; \; \; \; \; \;  if \; \{\varepsilon : \; P_{X}(x,v,\varepsilon) \; is \; true      \}= \emptyset. \\
                        \end{array}
                        \right.$
\end{center}
is well defined and definable.\\
Now take a cell decomposition $\mathcal{D}$ of $\mathbb{R}^{n}\times \mathbb{R}^{m} \times \mathbb{R}$ compatible with $\mathbb{R}^{n} \times \mathbb{R}^{m} \times \{0 \}$, $\mathcal{B}\times \{0 \}$ and $\Gamma _{F}$ (the graph of $F$). \\
If $\mathcal{D}_{n+m}$ is the cell decomposition of $\mathbb{R}^{n}\times \mathbb{R}^{m}$ determined by $\mathcal{D}$ (means that $\pi (\mathcal{D})=\mathcal{D}_{n+m}$ ), then for any $C\in \mathcal{D}_{n+m}$ of maximal dimension, some of the bands over $C$ (at least one) are in $\mathcal{X}$ and the others in $\mathcal{X}^{c}$. Hence by Lemma~\ref{lemma1} we can find a cell decomposition $\mathcal{D}'_{n+m}$ finer than $\mathcal{D}_{n+m}$ such that for all $C\in \mathcal{D}'_{n+m}$ with $dim(C)=n+m$ there is some box $[a_{C},b_{C}]\subset \mathbb{R}_{+}^{\ast}$ such that: $C\times [a_{C},b_{C}] \subset \mathcal{X}$.\\

Now since $dim(\mathcal{B}_{x})<m$ for all $x\in \mathbb{R}^{n}$, we can apply Lemma~\ref{lemma1} again and find a partition $\mathcal{P}$ of $\mathbb{R}^{n}$ with the fact that for all $P\in \mathcal{P}$ there is a box $B_{P}=[v_{P}(1),v'_{P}(1)]\times ...\times [v_{P}(m),v'_{P}(m)]\subset \mathbb{R}^{m}$ such that $P\times B_{P}$ is included in some band $C_{P}$ of $\mathcal{D}'_{n+m}$, hence:

\begin{center}
 $P\times B_{P} \times [a_{C_{P}},b_{C_{P}}]\subset \mathcal{X}$.
\end{center}
Finally we can take $\mathbb{A}=\{ v_{P}=(v_{P}(1),...,v_{P}(m) ) \}_{P\in \mathcal{P}}$ and $\varepsilon _{0}=min _{P\in \mathcal{P}} \{ a_{C_{P}} \}$.

\end{proof}
Now we define $P_{X}(x,v,\varepsilon)$ as follows:

\begin{center}
$P_{X}(x,v,\varepsilon)$ is true $\Leftrightarrow$ $\pi _{v}$ is $(\varepsilon,p)$-weak regular at $x$ with respect to $X$. 
\end{center}
It's obvious that $P_{X}$ satisfies $(1)$ of the precedent lemma. Let's prove that also $(2)$ holds true for $P_{X}$. We have:

\begin{center}
	\begin{eqnarray*}
		\mathcal{X}&=&\{ (x,v,\varepsilon) \; ; \; \pi _{v} \; is \; \varepsilon -weak \; regular \; at \; x \; with \; respect \;  to X  \} \\
		&=&\mathcal{X}_{1}\cup \mathcal{X}_{2}	
	\end{eqnarray*}
\end{center}
where $\mathcal{X}_{1}$ and $\mathcal{X}_{2}$ are defined below.

\begin{center}
	\begin{eqnarray*}
	\mathcal{X}_{1}&=&\{ (x,v,\varepsilon) \; ; \; C_{\varepsilon}(x,v)\cap X=\emptyset  \} \\
		&=&\{ (x,v,\varepsilon) \; ; \; \forall t\in \mathbb{R}^{\ast} \; and \; \forall v' \; such \; that \; \parallel v-v' \parallel <\varepsilon \; :  \; x+t(v',1)\notin X  \}
	\end{eqnarray*}
\end{center}
Hence $\mathcal{X}_{1}$ is defined by a first order formula, and therefore $\mathcal{X}_{1}$ is a definable set.  
\begin{center}
$\mathcal{X}_{2}=\{ (x,v,\varepsilon) \; : \; such \; that \;  C_{\varepsilon}(x,v)\cap X\neq \emptyset \; ,C_{\varepsilon}(x,v)\cap X \subset Reg^{p}(X) \;,and \; \forall y \in C_{\varepsilon}(x,v)\cap X,\; there\; is\;  t\in \mathbb{R}^{\ast}\; and \; v'\in B(v,\varepsilon ) \; with \;  y=x+t(v',1) \;  and \; the \;  line \;  through \;  x \; directed \; by \; (v',1) \; is \; transverse \; to \; X \; at \; y \;  \}$. 
\end{center}
Hence $\mathcal{X}_{2}$ is also defined by a first order formula, thus it's a definable set. Finally $\mathcal{X}$ is definable set. Let's prove that $dim(\mathcal{B}_{x})<n-1$ for all $x\in \mathbb{R}^{n}$. Assume that this is not true. For $x\in \mathbb{R}^{n}$ and $v\in \mathbb{R}^{n-1}$ we denote by $L(x,v)$ the line that passes through $x$ and is  directed by $(v,1)$. For all $x\in \mathbb{R}^{n}$ we have:

\begin{center}
$\mathcal{B}_{x}=\mathcal{B}_{x}^{1} \cup \mathcal{B}_{x}^{2}$,
\end{center}
where: 

\begin{center}
$\mathcal{B}_{x}^{1}= \{v\in \mathbb{R}^{n-1} \; : \; such \; that \; L(x,v) \; is \; not \; transverse \; to \; Reg^{p}(X)     \}$
\end{center}
and: 

\begin{center}
$\mathcal{B}_{x}^{2}=\{ v\in \mathbb{R}^{n-1} \; : \;  L(x,v)\cap Sing^{p}(X)\neq \emptyset \;     \}$.
\end{center}
We define the definable map:

\begin{center}
$\phi : \mathbb{R}^{n-1}\times \mathbb{R}^{\ast}\longrightarrow \mathbb{R}^{n}$\\
$(v,t)\mapsto \phi
 (v,t)=x_{0}+t(v,1)$.
\end{center}
A simple calculation shows that $\phi$ is a submersion, hence it is a local diffeomorphism, but since $dim(Sing^{p}(X))<(n-1)$, we can deduce that: 

\begin{center}
 $dim(\phi ^{-1}(Sing^{p}(X)))<n-1$,
 \end{center} 
therefore this means that: 

\begin{center}
$dim(\pi (\phi ^{-1}(Sing^{p}(X))))<n-1$,
\end{center}
hence: 

\begin{center}
$dim(\mathcal{B}_{x}^{2})<n-1$.
\end{center}
But by our assumption there must be an $x_{0}\in \mathbb{R}^{n}$ such that : $dim(\mathcal{B}_{x_{0}}^{1})=n-1$, and 

\begin{center}
$\mathcal{B}_{x_{0}}^{1}=\{v\in \mathbb{R}^{n-1} \; : \;such \; that \; \phi (v,.) \; is \; not \; transverse \; to \; Reg^{p}(X)     \}$.
\end{center}
Finally, by Lemma~\ref{lemma2} we deduce that $dim(\mathcal{B}_{x_{0}}^{1})<n-1$, thus this is a contradiction with $dim(\mathcal{B}_{x_{0}}^{1})=n-1$. This complete the proof of Theorem~\ref{theorem2}.

\end{proof}

\begin{rmk}
\begin{itemize} 
\item[(i)] By the proof of Theorem~\ref{theorem2} we can require $\{v_{1},...,v_{k} \}$ to be from an open subset of $(\mathbb{R}^{n-1})^{k}$. 
\item[(ii)] If $dim(X)<(n-1)$, then we can see that $\mathcal{X}_{2}=\emptyset$, hence this means that we can find a finite set of projections $\{v_{1},...,v_{k}  \}\subset \mathbb{R}^{n-1}$ and $\varepsilon >0$ such that for any $x\in \mathbb{R}^{n}$ there is an $i$ such that:

\begin{center}
$C_{\varepsilon}(x,v_{i})\cap X=\emptyset$.
\end{center}
\end{itemize}
\end{rmk}

\begin{question} From the proof of Theorem~\ref{theorem2} we can minimize the number of the projections by: 
\begin{center}
$k_{m}=min\{k\in \mathbb{N} \; : \; k=\mid \mathcal{P} \mid \; where \; \mathcal{P} \; is \; a \; partition  \; of \; \mathbb{R}^{n} \; such \; that \; for \; all \; P\in \mathcal{P} \; there \; is \; a \; box \; B_{P}=[v_{P}(1),v'_{P}(1)]\times ...\times [v_{P}(m),v'_{P}(m)]\subset \mathbb{R}^{m} \;  such \; that \; P\times B_{P} \; is \; included \; in \; some \; \; band \; C_{P} \; of \; \mathcal{D}'_{n+m} \;
     \}$.
\end{center}
Clearly this number depends by $n$ and $X$, therefore we ask the question: \\
 Is there a way to prove that the number of projections can be chosen independently of $X$?
 
 \end{question}
 
\begin{question} Can the set of projections be chosen in an open definable dense subset of $(\mathbb{R}^{n-1})^{k}$?
\end{question}

\subsection{Counter example in non-polynomially bounded o-minimal structures :}
Fix $\mathcal{A}$ a non-polynomially bounded o-minimal structure on $(\mathbb{R},+,.)$. Assume that the Regular projection theorem is true for this structure, and consider the set $X$ defined by:
\begin{center}
$X=X_{1} \cup X_{2}$,
\end{center}

where:
\begin{center}
$X_{1}=\{(x,ax^{a+1},x^{a+1}) \; : \; x> 0 \; and \;  a \in \mathbb{R}  \}$\\
$X_{2}=\{(x,-ax^{a+1},x^{a+1}) \; : \; x> 0 \; and \;  a \in \mathbb{R}  \}$.
\end{center}
For all $p\in \mathbb{N}^{\ast}$ we have : $Reg^{p}(X)=X\setminus \{ (x,0,x) \; : \; x>0  \}$, and the connected components of $Reg^{p}(X)$ are $C^{p}$-manifolds.

By Miller's Dichotomy Theorem\cite{C2}, the graph of the exponential function $x \longrightarrow exp(x)$ is a definable set, and therefore $X$ is a definable set.\\
Take the germ of the curve $x(s) = (s, 0, 0)$. By cell decomposition and the assumption there are $\delta > 0$, a vector $v = (v_{1} , v_{2} ) \in \mathbb{R}^{2}$ , $\varepsilon > 0$ and $C >0$ such
that for all $s \in [0,\delta [$ we have one of the following two cases:

\begin{itemize}

\item[$(1)$] $C_{\varepsilon} (x(s), v) \cap X =\emptyset$ .

\item[$(2)$] $C_{\varepsilon} (x(s), v) \cap X=\sqcup A _{f_{s}^{i}}$, where $f _{s} ^{i} : B(v,\varepsilon) \longrightarrow \mathbb{R}^{\ast} $ are $C^{p}-$regular definable functions such that for all $s$,
$i$, and $\lambda \in B(v,\varepsilon)$ we have: 

\begin{center}
$\frac{\parallel grad \; f_{s}^{i}(\lambda) \parallel}{\mid f_{s}^{i}(\lambda) \mid }\leq C$.
\end{center}
We are interested in the second case, because for $s$  small enough $C_{\varepsilon} (x(s), v) \cap X \neq \emptyset$. To obtain a contradiction, let's prove the next two facts:

\end{itemize}

\textbf{Fact 1 :} For $\lambda =(\lambda _{1},\lambda _{2}) \in B(v,\varepsilon )$, the functions $s\mapsto f_{s}^{i}(  \lambda )$ are characterized by the functional equation:

\begin{center}
 $f_{s}^{i}(\lambda)=(s+\lambda _{1} f_{s}^{i}(\lambda))^{\pm \lambda _{2}+1}$.
\end{center}
Indeed, take $s\in [0,\delta [$. \\
Since $\; \; \; x(s)+f_{s}^{i}(\lambda)(\lambda,1) \in C_{\varepsilon}(x(s),v) \cap X$, there is some $x \in \mathbb{R}_{+}$ and  $a\in \mathbb{R}$ such that:
\begin{center}
$x(s)+f_{s}^{i}(\lambda)(\lambda,1)=(x,\pm ax^{a+1},x^{a+1})$.
\end{center}
Hence:
\begin{center}
$\; \; \; \; \; \; \; \; \; \; x=s+f_{s}^{i}(\lambda)\lambda _{1}$\\
$f_{s}^{i}(\lambda)=x^{a+1}$\\
$\; \; \; \; \; \; \; \; \pm ax^{a+1}=\lambda _{2}x^{a+1}$
\end{center}
and this implies:
\begin{center}
$\lambda _{2}=\pm a \; \; $ and $\; \; \; f_{s}^{i}(\lambda)=(s+\lambda _{1} f_{s}^{i}(\lambda))^{\pm \lambda _{2}+1}$.
\end{center}
But since the sets $A_{f_{s}^{i}}$ are connected definable $C^{p}$-manifolds and the connected components of $Reg^{p}(X)$ are $C^{p}$-definable manifolds we deduce that the sign $\pm$ on $\lambda _{2}$ depends only on $i$.\\
Reciprocally, take $\lambda \in B(v , \varepsilon )$ and 
take a definable continuous function $s\mapsto t_{s} \in \mathbb{R}^{\ast}_{+}$ such that:
\begin{center}
$t_{s}=(s+\lambda _{1}t_{s})^{\pm \lambda _{2}+1}$.
\end{center}
Fix $s\in [0,\delta[$, then we have:
\begin{center}
$x(s)+t_{s}(\lambda,1) \in C_{\varepsilon}(x(s),v)$.
\end{center}
Let's take
\begin{center}
$x=s+\lambda _{1}t_{s}$ and $a=\pm \lambda _{2}$.
\end{center}
Then
\begin{center}
$x(s)+t_{s}(\lambda,1)=(x,\pm ax^{a+1},x^{a+1})$.
\end{center}
Hence
\begin{center}
 $x(s)+t_{s}(\lambda,1) \in C_{\varepsilon}(x(s),v) \cap X$.
 \end{center}
This implies that there is $i_{s}$ and $\lambda ' \in B(v,\varepsilon )$ such that:
\begin{center}
$x(s)+t_{s}(\lambda,1)=x(s)+f_{s}^{i_{s}}(\lambda ')(\lambda ',1)$.
\end{center}
And this implies that $\lambda=\lambda '$ and $t_{s}=f_{s}^{i_{s}}(\lambda)$.\\
By the fact that $f_{s}^{1}<....<f_{s}^{k}$, and the functions $s\mapsto t_{s}$ and $s\mapsto f_{s}^{i}(\lambda )$ are continuous and definable, we deduce that $i_{s}$ doesn't depend on $s$, hence $s\mapsto t_{s}$ is one of the functions $s\mapsto f_{s}^{i}(\lambda )$. \\
This finishes the proof of \textbf{Fact 1}.

\vspace*{0.5cm}
\textbf{Fact 2:} 

\begin{itemize}

\item[($\bullet$)] If $\lambda =(\lambda _{1}, \lambda _{2}) \in B(v,\varepsilon )$ such that  $\lambda _{2} > 0$ , then there is a solution $s\mapsto t_{s}$ of the equation:
\begin{center}
$t_{s}=(s+\lambda _{1}t_{s})^{\lambda _{2}+1}$
\end{center}
such that $lim _{s\mapsto 0} t_{s}=0$.\\
For that, we define the function $F:\mathbb{R}^{2}\longrightarrow \mathbb{R}$ by:
\begin{center}
$F(s,t)=\left\{
                        \begin{array}{ll}
                        t-(s+\lambda _{1}t)^{\lambda _{2}+1} \; \; \; \; \; \;  if \; s+\lambda _{1}t > 0 \\
                        t \; \; \; \; \; \; \; \; \; \; \; \; \; \; \; \; \; \; \; \; \; \; \; \; \; \; \; \; \; \; \;  if \; s+\lambda _{1}t \leqslant  0. \\
                        \end{array}
                        \right.$
\end{center}
$F$ is a $C^{1}$ function and we have:

\begin{center}

                       $ F(0,0)=0$  and $  \frac{\partial F}{\partial t}(0,0)=1$.

\end{center}
Hence by the Implicit Function theorem, there is a continuous function $s\mapsto t(s)$ such that $t(0)=0$ and $F(s,t(s))=0$.\\

 \item[($\bullet$)] If $\lambda \in B(v,\varepsilon )$ such that  $\lambda _{2} <0 $ , then there is a solution $s\mapsto t_{s}$ for the equation :
\begin{center}
$t_{s}=(s+\lambda _{1}t_{s})^{-\lambda _{2}+1}$
\end{center}
such that $lim _{s\mapsto 0} t_{s}=0$. We use the same argument as in the first case by applying the Implicit Function theorem to the function: 

\begin{center}
$F(s,t)=\left\{
                        \begin{array}{ll}
                        t-(s+\lambda _{1}t)^{-\lambda _{2}+1} \; \; \; \; \; \;  if \; s+\lambda _{1}t > 0 \\
                        t \; \; \; \; \; \; \; \; \; \; \; \; \; \; \; \; \; \;  \; \; \; \; \; \; \; \; \; \; \; \; \; \; \;  if \; s+\lambda _{1}t \leqslant  0. \\
                        \end{array}
                        \right.$
\end{center}
\end{itemize}

\vspace{0.8cm}

 Now we will discuss the projection in two cases:
 
 \begin{itemize}
 
\item[case 1 :] Assume that $v_{2}\geqslant 0$, and take $\lambda =(\lambda _{1},\lambda _{2})\in B(v,\varepsilon)$ such that $\lambda _{2}>0$. From \textbf{Facts 1 and 2} there is an $i_{0}$ such that $lim_{s\rightarrow 0}f_{s}^{i_{0}}(\lambda)=0$ and:

\begin{center}
$f_{s}^{i_{0}}(\lambda)=(s+\lambda _{1} f_{s}^{i_{0}}(\lambda))^{\lambda _{2}+1}$.
\end{center}
Therefore we have:
\begin{center}
$\frac{\partial f_{s}^{i_{0}}}{\partial \lambda _{2}}(\lambda )=(ln(s+\lambda _{1} f_{s}^{i_{0}}(\lambda))+\frac{(\lambda _{2}+1)\lambda _{1}\frac{\partial f_{s}^{i_{0}}}{\partial \lambda _{2}}(\lambda )}{s+\lambda _{1} f_{s}^{i_{0}}(\lambda)})f_{s}^{i_{0}}(\lambda) $.
\end{center}
Then we have:
\begin{center}
$\frac{\frac{\partial f_{s}^{i_{0}}}{\partial \lambda 2}(\lambda )}{f_{s}^{i_{0}}(\lambda)}(1-(\lambda _{2}+1)\lambda _{1}(s+\lambda _{1}f_{s}^{i_{0}}(\lambda))^{\lambda _{2}})=ln(s+\lambda _{1}f_{s}^{i_{0}}(\lambda))$.
\end{center}
Hence, since $\lambda _{2} > 0$ and $lim_{s\mapsto 0}f_{s}^{i_{0}}(\lambda)=0$, we deduce that:

\begin{center}
$\vert \frac{\frac{\partial f_{s}^{i_{0}}}{\partial \lambda 2}(\lambda )}{f_{s}^{i_{0}}(\lambda)} \vert \longrightarrow \infty \; \; $ when $\; \; s\longrightarrow 0$.
\end{center}
Thus, this is a contradiction with the fact that: $\vert \frac{\frac{\partial f_{s}^{i_{0}}}{\partial \lambda 2}(\lambda )}{f_{s}^{i_{0}}(\lambda)} \vert \leq C \; \; \forall s $. 
\\

\item[case 2 :] Assume that $v_{2}<0 $, and take $\lambda =(\lambda _{1},\lambda _{2})\in B(v,\varepsilon)$ such that $\lambda _{2}<0$. From \textbf{Fact 1 and 2} there is an $i_{0}$ such that $lim_{s\rightarrow 0}f_{s}^{i_{0}}(\lambda)=0$ and:

\begin{center}
$f_{s}^{i_{0}}(\lambda)=(s+\lambda _{1} f_{s}^{i_{0}}(\lambda))^{-\lambda _{2}+1}$
\end{center}
We have:
\begin{center}
$\frac{\partial f_{s}^{i_{0}}}{\partial \lambda 2}(\lambda )=(-ln(s+\lambda _{1} f_{s}^{i_{0}}(\lambda))+\frac{(-\lambda _{2}+1)\lambda _{1}\frac{\partial f_{s}^{i_{0}}}{\partial \lambda 2}(\lambda )}{s+\lambda _{1} f_{s}^{i_{0}}(\lambda)})f_{s}^{i_{0}}(\lambda) $
\end{center}
Then:
\begin{center}
$\frac{\frac{\partial f_{s}^{i_{0}}}{\partial \lambda 2}(\lambda )}{f_{s}^{i_{0}}(\lambda)}(1-(-\lambda _{2}+1)\lambda _{1}(s+\lambda _{1}f_{s}^{i_{0}}(\lambda))^{-\lambda _{2}})=-ln(s+\lambda _{1}f_{s}^{i_{0}}(\lambda))$
\end{center}
Hence, since $- \lambda _{2} > 0$ and $lim_{s\mapsto 0}f_{s}^{i_{0}}(\lambda)=0$, we deduce that:

\begin{center}
$\vert \frac{\frac{\partial f_{s}^{i_{0}}}{\partial \lambda 2}(\lambda )}{f_{s}^{i_{0}}(\lambda)} \vert \longrightarrow \infty \; \; $ when $\; \; s\longrightarrow 0$
\end{center}
Hence this is a contradiction with the fact that: $\vert \frac{\frac{\partial f_{s}^{i_{0}}}{\partial \lambda 2}(\lambda )}{f_{s}^{i_{0}}(\lambda)} \vert \leq C \; \; \forall s$. 

 \end{itemize}

\begin{flushright}
$\square$
\end{flushright}

\subsection{Regular projection theorem in polynomially bounded structures. }

Let $\mathcal{A}$ be an o-minimal structure on $\mathbb{R}$ and let $\mathcal{K}$ be the field of exponents of $\mathcal{A}$, means that:

\begin{center}
$\mathcal{K}=\{r\in \mathbb{R} \; : \; x \mapsto x^{r} \; is \; definable \; in \; \mathcal{A} \;      \}$.
\end{center}
 Before beginning the proof of the theorem we recall a few results:

\begin{lem}\label{lemma4} \textbf{(Piecewise Uniform Asymptotics).} Assume that $\mathcal{A}$ is polynomially bounded. Let $f:\mathbb{R}^{m}\times \mathbb{R}\longrightarrow \mathbb{R}$ be definable. Then there is a finite $S\subset \mathcal{K}$ such that for each $y\in \mathbb{R}^{m}$ either the function $t\mapsto f(y,t)$ is ultimately equal to $0$ or there exists $r\in S$ such that $lim _{t\mapsto +\infty}\frac{f(y,t)}{t^{r}}\in \mathbb{R}^{\ast}$.

\end{lem}

\begin{proof}
See [\cite{C1}, Proposition5.2]
\end{proof}

\begin{lem}\label{lemma5} Let $ U\subset \mathbb{R}^{k}$  be a non empty definable open set and $M:U\times ]0,\alpha[ \longrightarrow \mathbb{R}$ be a $C^{1}$ definable map. Suppose there exists $K>0$ such that $\mid M(y,t) \mid \leq K$, for $(y,t)$. Then there is a closed definable subset $F$ of $U$ with $dim(F)<dim(U)$ and continuous definable functions $C, \tau : U\setminus F \longrightarrow \mathbb{R}^ {\ast} _{+}$, such that for all $y\in U \setminus F$ we have:
\begin{center}
$\parallel D_{y}M(y,t) \parallel \leq C(y)$ for all $t\in ]0,\tau (y)[$. 
\end{center}
\end{lem}
 
 \begin{proof} See [\cite{L1}, Lemma1.8]  
\end{proof}

\begin{lem}\label{mono} \textbf{(Monotonicity theorem).}
Let $f:]a,b[\longrightarrow \mathbb{R}$ be a definable map. Then there are points $a=a_0 < a_1 <..<a_k =b$ such that $f\restriction_ ]a_i , a_i+1[ $ is either constant or strictly monotone on $]a_i , a_i+1[$ for each $i=0,...,n-1$. 
\end{lem} 
\begin{proof}
See \cite{M} or \cite{VD}.
\end{proof}

\begin{lem}\label{curve selection} \textbf{(The curve selection Lemma).}
Let $A$ be a definable subset of $\mathbb{R}^{n}$ and $a\in \overline{A}\setminus A$. Let $p\in \mathbb{N}$. Then there exists a $C^p$ definable curve $\gamma : ]0,1[\longrightarrow A\setminus \{a\}$, such that $lim_{t\mapsto 0^+}\gamma (t)=a$. 
\end{lem}
\begin{proof}
See \cite{M} or \cite{VD}.
\end{proof}
In all the rest of this section we assume that the o-minimal structure $\mathcal{A}$ is polynomially bounded. 

\begin{lem} \label{lemma6}

Take $\Omega$ a definable open neighborhood of $0$ in $\mathbb{R}\times \mathbb{R}^{m}$, and a definable function:

\begin{center}
$f:\Omega \longrightarrow \mathbb{R}$\\
$\; \; \; \; \; \; \; \; (x,y)\mapsto f(x,y)$
\end{center}
such that: $f^{-1}(0)\subset (\Omega \cap \mathbb{R}^{m})$ and $f$ is $C^{1}$ with respect to $y$ on $\Omega \setminus \mathbb{R}^{m}$. Then there exists a definable set $W\subset (\Omega \cap \mathbb{R}^{m})$ with $dim(W)<m$, such that for all $(0,y) \in \Omega \setminus W$ there is some $\varepsilon >0$ and $C>0$ such that we have on $B((0,y),\varepsilon )\cap (\Omega \setminus \mathbb{R}^{m})$:

 \begin{center}
 $\parallel D_{y}f \parallel \leqslant C \mid f \mid $.
  \end{center}

\begin{proof}

 We can assume that $f\geqslant 0$. Let's assume that the statement is not true, then we can find $O\subset (\Omega \cap \mathbb{R}^{m})$ with $dim(O)=m$ and for all $(0,y)\in O$ we have:

\begin{center}
 $ (\ast) \; \; \; \; \; \; \; \; \; \; \; \; \; \; \; \; lim _{x\mapsto 0^{+}}\frac{\parallel D_{y}f (x,y) \parallel}{\parallel f(x,y) \parallel}=+\infty    $.
\end{center}
Now we can find an open definable set $U\subset O$ and $\alpha >0$ such that $ ]0,\alpha [\times U \subset \Omega$. Since $f^{-1}(0)\subset (\Omega \cap \mathbb{R}^m)$,  by Lemma~\ref{lemma4} there exist an open set  $B\subset  U$ and $r\in \mathbb{R}$ such that:

\begin{center}
 $ f(x,y)=c(y)x^{r}+\phi (x,y)x^{r}$, for $y\in B$ and $\alpha ' >x>0$ for some $\alpha ' < \alpha $,
 \end{center} 
where $c$ is a definable function on $B$ with $c(y)\neq 0$ for all $y\in B$, $\phi$ definable such that $lim _{x\mapsto 0}\phi
 (x,y)=0$ for all $y$ in B. Therefore for any $y\in B$ and $x\in ]0,\alpha ' [$ we have:
 
\begin{center}
$\frac{\parallel D_{y}f (x,y) \parallel}{\parallel f(x,y) \parallel}=\frac{\parallel Dc(y)+D_{y}\phi (x,y)   \parallel}  {\parallel c(y)+\phi (x,y)     \parallel}\leqslant \frac{\parallel Dc(y) \parallel}{\parallel c(y)+\phi (x,y)     \parallel}+ \frac{\parallel D_{y}\phi (x,y) \parallel }{\parallel c(y)+\phi (x,y)     \parallel}$.
\end{center}
But by Lemma~\ref{lemma5} and the fact that $lim _{x\mapsto 0}\phi
 (x,y)=0$ for all $y$ in B, we can shrink $B$ and $]0,\alpha '[$ and assume that there is a definable function $y\mapsto M(y)>0$ such that for any $(x,y)\in  ]0,\alpha '[ \times B$ we have:

\begin{center}
$\parallel   D_{y}\phi (x,y)    \parallel \leqslant M(y)$
\end{center}
Hence for any $y\in B$: 

\begin{center}
  $lim _{x\mapsto 0^{+}}\frac{\parallel D_{y}f (x,y) \parallel}{\parallel f(x,y) \parallel} \leqslant \frac{\parallel Dc(y) \parallel}{\parallel c(y)     \parallel}+ \frac{M(y)}{\parallel c(y)\parallel} $
  \end{center}  
 and this contradicts $(\ast)$.

\end{proof}  
 
\end{lem}

\begin{lem} \label{lemma7} Let $f: [0,\varepsilon [ \times \Omega \longrightarrow \mathbb{R}$ be a definable function, $C^{1}$ with respect to $y$ on $]0,\varepsilon [ \times \Omega$, where $\Omega$ is an open subset of $\mathbb{R} ^{m}$. Then we can find a definable subset $W\subset \Omega$ with $dim(W)<m$ and such that for every $y_0 \in \Omega \setminus W$ there are $r>0$, $\varepsilon '>0$, and $C>0$ such that we have: 

\begin{center}
$\parallel D_{y_{0}}f(x,y) \parallel \leqslant C \mid f(x,y) \mid $ for all $(x,y)\in  ]0,\varepsilon '[ \times B(y_0,r)$.
\end{center}
 \begin{proof}
 
   Take $Z=f^{-1}(0)$, and take $\mathcal{C}$ a cell decomposition of $  [0,\varepsilon [ \times \Omega$ compatible with $\{ 0 \} \times \Omega$ and $Z\cap   (]0,\varepsilon [ \times \Omega) $. Hence $\Omega$ is a finite disjoint union of cells of $\mathcal{C}$. Take $O$ one of these cells with dimension equal to $m$, the cells of dimension smaller than $m$ will be chosen to be in the set $W$. We discuss two cases:

\begin{itemize}

\item[case1:] There is a cell $B\subset Z$ of maximal dimension (i.e $dim(B)=m+1$) such that $O\subset   \overline{B}$. In this case at each point $(0,y)\in O$  we can find a neighborhood of $(0,y)$ in $ ]0,\varepsilon [ \times \Omega$, where $f\equiv 0$ on this neighborhood, hence in this neighborhood $\parallel D_{y}f \parallel \equiv 0$, then the result holds at $(0,y)$ for any $C>0$.
\item[case2:]  There is no cell $B\subset Z$ of maximal dimension such that $O\subset \overline{B}$. In this case by Lemma~\ref{lemma6} we can find a definable set $W_{O} \subset O$ with $dim(W_{O})<m$ and such that for all $(0,y)\in O\setminus W_{O}$ we can find $r_{y}>0$, $\varepsilon _{y}>0$, and $C_{y}>0$,  such that $\parallel D_{y}f \parallel \leqslant C_{y} \mid f \mid $ on $ ]0,\varepsilon _{y}[ \times B(y,r_{y})$.

\end{itemize}

For cells like in case.1 we define $W_{O}:=\emptyset$. So finally the set:

\begin{center}
$W:=(\bigcup _{O, \; with \; dim(O)=m}W_{O})\cup (\bigcup _{dim(O)<m}O)$
\end{center}
satisfies the required properties.

\end{proof}

\end{lem}

\begin{defi}
Let $X$ be a definable subset of a definable manifold $N$, and $\Omega$ a definable open subset of $\mathbb{R}^{m}$. Take $f:X\times \Omega \longrightarrow \mathbb{R}$ , $(x,y)\mapsto f(x,y) $ a definable function and $C^{1}$ with respect to $y$. We say that $f$ is $X-$rectifiable with respect to $y$, if we can find a definable partition $\mathcal{P}$ of $X$, and $c>0$ such that for every $D\in \mathcal{P}$ there is a box $B_{D}\subset \Omega$ such that:

\begin{center}
$\forall x\in D$ and $\forall y\in B_{D}$ : $\parallel D_{y}f(x,y) \parallel \leqslant c \mid f(x,y) \mid $.
\end{center}
\end{defi}

\begin{rmk} It is obvious that if there is a finite definable cover $(X_{i})_{i}$ of $X$ such that for each $i$ $f_{\mid X_{i} \times \Omega}$ is $X_{i}-$rectifiable with respect to $y$, then $f$ is $X-$rectifiable with respect to $y$.
\end{rmk}

\begin{lem} \label{lemma8}

Let $X$ be a definable subset of $\mathbb{S}^{n}$, $\Omega$ be an open definable subset of $\mathbb{R}^{m}$, and $x_{0}\in \partial X$. Let $f:X\times \Omega  \longrightarrow \mathbb{R}$, $(x,y)\mapsto f(x,y)$ be a definable function, $C^{1}$ with respect to $y$. Then there is a neighborhood $U$ of $x_{0}$ in $\mathbb{S}^{n}$ such that $f_{\mid U\cap X \times \Omega}$ is $U\cap X$-rectifiable with respect to $y$ , recall that it means we can find $\alpha >0$, $c>0$, $\mathcal{C}$ a definable partition of $X$, and $\mathcal{B}$ a finite collection of boxes in $\Omega $ such that for every $C\in \mathcal{C}$  with $x_{0} \in \overline{C}$, there is $B_{C}\in \mathcal{B}$ such that:

\begin{center}
$\forall x\in C$ with $ d(x,\; x_{0}) <\alpha$, $\forall y\in B_{C}$ : $\parallel D_{y}f(x,y) \parallel \leqslant c \mid f(x,y) \mid $.
\end{center}
\end{lem}

\begin{proof}
Induction on $m$.

\begin{itemize}

\item[$\bullet$] \textbf{The case of $m=1$:}

For the case of $n=1$, i.e $X\subset \mathbb{S}^{1}$,  by choosing a good coordinate ( we can choose a stereographic projection ) we can reduce this to the case of 
Lemma~\ref{lemma7} and this finishes the proof of this case. Hence by contradiction we assume that $n>1$, $\Omega =]a,b[$,  and there is no neighborhood $U$ of $x_{0}$ in $\mathbb{S}^{n}$ such that $f_{\mid U\cap X \times \Omega}$ is $U\cap X$-rectifiable with respect to $y$. We will discuss the next two cases:

\begin{itemize}

\item[case1:] Assume that $f^{-1}(0)=\emptyset$. We define the function:

\begin{center}
$g:X\times \Omega \longrightarrow \mathbb{R}_{+}$\\
$\; \; \; \; (x,y)\mapsto \frac{\parallel D_{y}f(x,y) \parallel}{\mid f(x,y) \mid}$.
\end{center}
Let's define the set:

\begin{center}
$O=\{ (x,y)\in X \times \Omega  \; : \; such \; that \; g(x,.) \; is \; monotone \; near \; y \;       \}$.
\end{center}
Since $O$ can be expressed by a first order formula, $O$ is a definable subset of $X \times \Omega $. Now take the set $\sum =(X \times \Omega )  \setminus O$, by Lemma~\ref{mono} we have for all $x\in X$:

\begin{center}
$dim(\sum_{x})=0  < dim (\Omega  )=1$
\end{center}
Hence by Lemma~\ref{lemma1} we can find a definable partition $\mathcal{C}_{1}$ of $X $ such that for all $C' \in \mathcal{C}_{1}$ there is a box $B_{C'}=]a_{c'},b_{c'}[ \subset \Omega $ such that:

\begin{center}
$C' \times B_{C'}  \cap \sum =\emptyset$
\end{center}
Then this means that for every $x\in C'$, $g(x,.)$ is monotone on $B_{C'}$, but we need that the monotonicity type does not depend on $x$ but only on $C'$. Let's choose $\mathcal{C}_{0}$ to be a partition of $X $ compatible with $\mathcal{C}_{1}$ and a collection of sets $\{ M_{C'}^{+},M_{C'}^{-},M_{C'}^{0} \}_{C'\in \mathcal{C}_{1}}$, where $M_{C'}^{+}$, $M_{C'}^{-}$, and $M_{C'}^{0}$ are defined by:

\begin{center}
$M_{C'}^{+}=\{ x\in C' \; : \; g(x,.) \; is \; increasing     \}$\\
$M_{C'}^{-}=\{ x\in C' \; : \; g(x,.) \; is \; decreasing     \}$\\
$M_{C'}^{0}=\{ x\in C' \; : \; g(x,.) \; is \; constant     \}$
\end{center}
Hence the monotonicity type of $g(x,.)$ depends only on the elements of $\mathcal{C}_{0}$, means that for any $C\in \mathcal{C}_0$ there is a box $B_C=]a_C,b_C[$ in $\Omega$ such that for any $x\in C$, $g(x,.)$ is monotone on $B_{C}$, with the same monotonicity for all $x\in C$ .Now let's apply our assumption to $\mathcal{C}_{0}$, then we can replace $X$ by an element $C\in \mathcal{C}_{0}$, with $x_{0} \in \overline{C}$ such that for every box $B$ in $\Omega $, $g$ is not bounded on $C\times B$. Take $D_{C}=[d_{C},e_{C}]\subset B_{C}$ and let's consider the graph of $g$ on $C\times D_{C}$:

\begin{center}
$\Gamma _{g}= \{ (x,y,g(x,y)) \; : \; y\in D_{C} \; , x\in C     \}\subset  \overline{C} \times D_{C}  \times (\mathbb{R}_{+} \cup \{ +\infty \})  $.
\end{center}
Consider $\overline{\Gamma}_{g}$ the closure of $\Gamma _{g}$ in $\overline{C}  \times D_{C}  \times (\mathbb{R}_{+} \cup \{ +\infty \})$. But since for any box $B$ in $D_C$, $g$ is not bounded on $C \times B$, we can find $y_{0}\in D_{C}$ such that:

\begin{center}
$(x_{0},y_{0},+ \infty)\in \overline{\Gamma _{g}}\setminus \Gamma _{g}$
\end{center}
By the curve selection Lemma there is a definable continuous curve $\gamma : [0,a[\longrightarrow \overline{\Gamma _{g}}$, with $\gamma (t)=(x(t),y(t),g(x(t),y(t)))$ and such that:

\begin{center}
$\gamma (0)=(x_{0},y_{0},\infty)$ and $\gamma (]0,a[)\subset \Gamma _{g}$.
\end{center}
Assume that $g(x,.)$ is increasing for all $x\in C$ (the other case is similar). In this case by Lemma~\ref{lemma7}, we can shrink $]0,a[$ and we can find a box $B$ in $]e_C,b_{C}[$ ( it is not empty because $e_{C}<b_{C}$ ) and $L>0$ such that for all $y\in B$, and for all $t\in ]0,a[$ we have:

\begin{center}
$g(x(t),y)\leqslant L$.
\end{center}
But we have also for all $y\in B$, and for all $t\in ]0,a[$: 

\begin{center}
 $g(x(t),y(t)) \leqslant g(x(t),y)<L$.
 \end{center} 
 
And this contradicts the fact that: 

\begin{center}
$lim _{t\mapsto 0}g(x(t),y(t))=+\infty$.
\end{center}

\vspace*{0.3cm}

\item[Case2:] We assume that $Z=f^{-1}(0)\neq \emptyset$. Take $\varepsilon >0$, define the set $A=\{x\in B(x_{0},\varepsilon)\cap X  \; : \; dim(Z_{x})=1  \}$. We have the following cases:

\begin{itemize}

\item[case.A:] Assume that $x_{0}\in  \overline{A}$, take $\mathcal{D}$ a definable partition of $B(x_{0},\varepsilon)\cap X$ compatible with $A$. Since $f$ is not $B(x_{0},\varepsilon)\cap X-$rectifiable with respect to $y$ near $x_{0}$, we can find $C\in \mathcal{D}$ with $x_{0}\in \overline{C}$ and such that for any partition $\mathcal{C}_{C}$ of $C$ there is $C'\in \mathcal{C}_{C}$ (with $x_{0}\in \overline{C}'$) such that for any box $B\subset \Omega$ there is no $L>0$ such that: 

\begin{center}
$\parallel D_{y}f(x,y) \parallel \leqslant L \mid f(x,y) \mid $ for $(x,y)\in C' \times B$ and $x$ in some neighborhood of $x_{0}$.
\end{center}
Here we discuss two subcases:

\begin{itemize}

\item[case.A.1:] Suppose $C\cap A=\emptyset$. Then by Lemma~\ref{lemma1} we can find a partition $\mathcal{C}_{C}$ of $C$ such that for every $C'\in \mathcal{C}_{C}$ there is a box $B_{C'}$ in $\Omega _{2}$ such that $(C' \times B_{C'}) \cap Z=\emptyset$,  hence $g$ is well-defined on $C'\times B_{C'}$. Then we may apply the same argument as in the proof of case 1 by considering for every $C'\in \mathcal{C}$ the function:

\begin{center}
$g:C'\times B_{C'} \longrightarrow \mathbb{R}_{+}$\\
$\; \; \; \; (x,y)\mapsto \frac{\parallel D_{y}f(x,y) \parallel}{\mid f(x,y) \mid}$.
 \end{center} 
\vspace*{0.3cm}

\item[case.A.2:] Suppose $C\subset A$. Then again by Lemma~\ref{lemma1} we can find a partition $\mathcal{C}_{C}$ of $C$ such that for every $C'\in \mathcal{C}_{C}$ there is a box $B_{C'}$ in $\Omega $ such that $(C' \times B_{C'}) \subset Z$ or $(C' \times B_{C'}) \cap Z=\emptyset$. If $(C' \times B_{C'}) \cap Z=\emptyset$, then we are again in the situation of case.A.1. If $(C' \times B_{C'}) \subset Z$, this gives that $f=0$ and $D_{y}f=0$ on $C' \times B_{C'}$ and this is a contradiction because we can choose any $L>0$ such that: 

\begin{center}
$\parallel D_{y}f(x,y) \parallel \leqslant L \mid f(x,y) \mid $ for $(x,y)\in C' \times B_{C'}$, and $x$ in some neighborhood of $x_{0}$.
\end{center}
\end{itemize}

\item[case.B:] Assume that $x_{0}\notin$  $\overline{A}$. Then we can separate $x_{0}$ from $A$ by some $B(x_{0},\varepsilon ')$ and proceed as in the case.A.2.  
  
\end{itemize}

\end{itemize}

\vspace{0.6cm}

\item[$\bullet$] \textbf{The case of $m>1$ :}  Assume that the statement of the lemma is true for any positive integer smaller than $m$. Fix $I=I_{1}\times ... I_{m}$ a box in $\Omega$, with $I_{i}=]a_{i},b_{i}[$ and denote by $I'=I_{2} \times ... \times I_{m}$. Apply the induction hypothesis ( the case of $m-1$ ) to the function:

\begin{center}
$f:(X\times I_{1})\times I' \longrightarrow \mathbb{R}$\\
$\; \; \; \; \; \; \; ((x,y_{1}),y')\mapsto f(x,(y_{1},y'))$.
\end{center}
Hence we can shrink $I_{1}$, and find a neighborhood $U$ of $x_{0}$, such that $f$ is $(X\cap U)\times I_{1}$-rectifiable with respect to $y'$, means that there is $L>0$ and a partition $\mathcal{P}$ of $(X\cap U)\times I_{1}$ ( we can choose it to be compatible with $X\cap U$ ) such that for any $C\in \mathcal{P}$ there is a box $B_{C}\subset I'$ and we have:

\begin{center}
$\forall (x,y_{1}) \in C$, $\forall y' \in B_{C}$ : $\parallel D_{y'}f(x,y) \parallel \leqslant L \mid f(x,y) \mid$.
\end{center}
Since $\mathcal{P}$ is compatible with $X\cap U$, we can find a partition $\mathcal{D}$ of $X\cap U$ and a definable maps $\phi _{1},..., \phi _{k_{D}}: D\longrightarrow I_{1}$ for every $D\in \mathcal{D}$, such that every $C\in \mathcal{P}$ is either the graph of one of these maps or a band between two graphs. Now take the set $\sum$ defined by:

\begin{center}
$\sum = \sqcup _{D\in \mathcal{D}} \bigcup _{i=1,...,k_{D}}\Gamma _{\phi _{l}} \subset (X\cap U) \times I_{1}$.
\end{center}
Since for all $x\in X\cap U$ we have $dim(\sum _{x})=0 < 1=dim(I_{1})$, by Lemma~\ref{lemma1} we can assume that the cells $C\in \mathcal{P}$ are of the form: 
 
 \begin{center}
  $C=D_{C}\times ]a_{C},b_{C}[$,
  \end{center} 
where $\{ D_{C} \}_{C\in \mathcal{P}}$ is a partition of $X\cap U$ and $]a_{C},b_{C}[\subset I_{1}$. Hence to complete the proof we need to show that for such $C\in \mathcal{P}$ with $x_{0}\in \overline{C}$, $f$ is $C$-rectifiable with respect to $y$ in some neighborhood of $x_{0}$. Now let's apply the case of $m=1$ to the function:

\begin{center}
 $f:(D_{C}\times B_{C}) \times ]a_{C},b_{C}[ \longrightarrow \mathbb{R}$
 $\; \; \; \; \; \; ((x,y'),y_{1})\mapsto f(x,y)$
 \end{center} 
We can find $U_{C}$ a neighborhood of $x_{0}$ in $X$ and we can shrink $B_{C}$, so that $f$ is $(U_{C}\cap D_{C})\times B_{C}$-rectifiable with respect to $y_{1}$, that there is $L'>0$ and a partition $\mathcal{P}_{C}$ of $(U_{C}\cap D_{C})\times B_{C}$ compatible with $U_{C}\times D_{C}$ such that for every $P\in \mathcal{P}_{C}$ there is a box $I_{P}\subset I_{1}$ and: 

\begin{center}
 $\forall (x,y')\in P$ and $\forall y_{1}\in I_{P}$ we have $\parallel  D_{y_{1}} f(x,y) \parallel \leqslant L' \mid  f(x,y) \mid$
 \end{center} 
Now by Lemma~\ref{lemma1} we can use the same argument as before and we can assume that every element  $P\in \mathcal{P}_{D}$ can be written in the form $P=M_{P} \times B_{P}$, where the collection $\{ M_{P} \} _{P} $ is a partition of $U_{C}\cap D_{C}$ and $\{ B_{P} \}_{P}$ are boxes in $B_{C}$. Therefore, finally, for any $M_{P}$ there is a box $I_{P}\times B_{P}\subset I$ such that on this box we have:

\begin{center}
 $\parallel  D_{y}f(x,y)  \parallel \leqslant (L+L') \mid f(x,y) \mid $.
 \end{center} 
This finishes the proof of the lemma.
\end{itemize}

\end{proof}

\begin{thm}
Let $A$ be definable subset of $\mathbb{R}^{n}$ such that $dim(A) \leqslant n-1$ and take $p\in \mathbb{N}^{\ast}$. Then there is $\varepsilon > 0$, $C>0$, and $\{v_{1},...,v_{k} \}\subset \mathbb{R}^{n-1}$ such that for every $x\in \mathbb{R}^{n}$ there is $i\in \{ 1,...,k \}$ such that $\pi _{v_{i}}$ is  $(\varepsilon,C,p)$- regular at $x$ with respect to $A$.
\end{thm}

\begin{proof}

 By Theorem~\ref{theorem2}, we can find a cell decomposition $\mathcal{C}=\{C_{1},...,C_{k} \}$ of $\mathbb{R}^{n}$, and $\varepsilon >0$ such that for every $i\in \{1,...,k \}$  there is $v_{i}\in \mathbb{R}^{n-1}$ such that $\pi _{v_{i}}$ is $(\varepsilon, p)$-weak regular at every point $x\in C_{i}$ with respect to $A$. Hence, for every $i\in \{ 1,...,k \}$ there are definable functions: 

\begin{center}
$f_{i,l}:C_{i}\times B(v_{i},\varepsilon)\longrightarrow  \mathbb{R}^{\ast} $\\
$\; \; \; (x,v)\mapsto f_{i,l} (x,v)$
\end{center}
such that each $f_{i,l}$ is $C^{p}$ with respect to $v$ and we have:
\begin{center}
$C_{\varepsilon}(x,v_{i})\cap X= \sqcup _{l}\{x+f_{i,l}(x,v )(v , 1) \; : \; v \in B(v_{i} ,\varepsilon )  \}$.
\end{center}
Now we will find a refinement of $\mathcal{C}$, constants  $\varepsilon ' < \varepsilon$ and $C>0$, and a $(C,\varepsilon ' , p)$-regular projections of $\mathbb{R}^{n}$ with respect to $A$. For this, it is enough to prove that for every $C_{i}\in \mathcal{C}$ the maps $f_{i,l}$ are $C_{i}-$rectifiable with respect to $v$ , means that there is $c>0$ and a definable partition $\mathcal{C}_{i}$ of $C_{i}$ such that for every $D\in \mathcal{C}_{i}$ there is a box $B_{D}\subset B(v_{i},\varepsilon )$ with:

\begin{center}
$\frac{\parallel D_{v}f_{i,l} \parallel }{\mid  f_{i,l} \mid} <c$ on $D\times B_{D}$. 
\end{center}
Take $C_{i} \in \mathcal{C}$, we have: 
\vspace{0.5cm}
\begin{itemize}

\item[$\bullet$] If $C_{i}$ is of dimension $0$, means that $C_{i}=\{x\}$ is a point, then by continuity of $\frac{\parallel D_{v}f_{i,l} \parallel }{\mid  f_{i,l} \mid}$ we can find a closed ball $\overline{B}(v_{i},\varepsilon ')\subset B(v_{i},\varepsilon)$ such that $\frac{\parallel D_{v}f_{i,l} \parallel }{\mid  f_{i,l} \mid}$ is bounded by some $c>0$ on this ball. Hence the maps $f_{i,l}$ are $C_{i}-$rectifiable with respect to $v$. 

\vspace{0.5cm}

\item[$\bullet$] Assume that $C_{i}$ is of dimension $dim(C_{i})>0$. We know that we have a natural definable embedding ( it is the inverse map of the stereographic projection, and it is semi-algebraic) of $\mathbb{R}^{n}$ in $\mathbb{S}^{n}$:

\begin{center}
$E:\mathbb{R}^{n} \longrightarrow  \mathbb{S}^{n}$.
\end{center}
We can replace $C_{i}$ by $E(C_{i})$ and $f_{i,l}$ by the maps $f_{i,l}\circ (E^{-1}, Id _{B(v_{i},\varepsilon)})$. Indeed if the maps $f_{i,l}\circ (E^{-1}, Id _{B(v_{i},\varepsilon)})$ are $E(C_{i})-$rectifiable then the maps $f_{i,l}$ are also $C_{i}-$rectifiable.  \\    

 Take the closure $\overline{C}_{i}$ of $C_{i}$ in $\mathbb{S}^{n}$ (it is a compact subset since $\mathbb{S}^{n}$ is compact ). Take $(U_{1},...,U_{k})$ an open definable cover of $\overline{C}_{i}$ given by: 

\begin{center}
  $U_{j}=\overline{C}_{i}\cap B(x_{j},r _{j})$ for  $j=1,...,k$, such that $x_{j}\in \overline{C}_{i}$ and $r_{j}>0$.
  \end{center}  
Now take $x\in \partial C _{i}= \overline{C} _{i}\setminus C_{i}$, hence there is a $j_{x}\in \{ 1,...,k \}$ such that $x\in U_{j_{x}}$. By applying  Lemma~\ref{lemma8} to the maps $f_{i,l}:(U_{j_{x}}\setminus \partial C_{i}) \times B(v_{i},\varepsilon) \longrightarrow \mathbb{R}$ (here $U_{j_{x}} \setminus \partial C_{i}=X$ and $B(v_{i},\varepsilon)=\Omega$), we can find a neighborhood $O_{x}$ of $x$ in $\overline{C}_{i}$ such that $f_{i,l}$ are $O_{x}\cap C_{i}-$rectifiable with respect to $v$ . Then by compactness of $\partial C _{i}$, we can choose a finite  cover $(O_{x_{1}},...,O_{x_{m}})$ of an open neighborhood of $\partial C _{i}$ in $\overline{C}_{i}$ such that $f_{i,l}$ are $(\bigcup _{s} O_{x_{s}})\cap C_{i}-$rectifiable with respect to $v$. Now since  $d(\partial C_{i} , C_{i} \setminus (\bigcup _{s} O_{x_{s}}))>0$, we can find a compact subset $K\subset C_{i}$ with $C_{i}\setminus (\bigcup _{s} O_{x_{s}}) \subset K$, hence  the functions $\frac{\parallel D_{v}f_{i,l} \parallel }{\mid  f_{i,l} \mid}$ are bounded on $C_{i} \setminus (\bigcup _{s} O_{x_{s}})$, therefore $f_{i,l}$ are $C_{i} \setminus (\bigcup _{s} O_{x_{s}})-$rectifiable with respect to $v$. Finally, since $(\bigcup _{s}O_{x_{s}} \cap C_{i}, C_{i}\setminus \bigcup _{s}O_{x_{s}})$ is a definable cover of $C_{i}$, we deduce that the functions $f_{i,l}$ are $C_{i}-$rectifiable with respect to $v$.

\end{itemize}

\end{proof}

\section{Existence of Regular covers}
Take $\mathcal{D}$ an o-minimal structure on $( \mathbb{R},+,.)$. Let $U$ be an open definable relatively compact subset of $\mathbb{R}^{n}$. By a regular cover of $U$, we mean a finite cover $(U_{i})$ by open definable sets, such that:  
\begin{itemize}
\item[$(1)$] each $U_{i}$ is homeomorphic to the open unit ball in $\mathbb{R}^{n}$ by a definable homeomorphism.
\item[$(2)$] there is a positive number $C$ such that for all $x$ in $\mathbb{R}^{n}$ we have :
\begin{center}
$d(x,\mathbb{R}^{n}\setminus U)\leq Cmax_{i}d(x,M\setminus U_{i})$.
\end{center}
\end{itemize}

\textbf{Example:} Take $U=\{(x,y)\in \mathbb{R}^{2} \; ; \; 0<x<1 \; and \; -1<y<1     \}$, $U$ is a definable set in any o-minimal structure on $\mathbb{R}$. The definable cover  $(U_{1},U_{2})$ of $U$ defined by:
\begin{center}
$U_{1}=\{(x,y)\in U \; ; \; y>-\frac{1}{2}  \}$\\
$U_{2}=\{(x,y)\in U \; ; \; y<\frac{1}{2}    \}$
\end{center}
is a regular cover. Indeed, take $p=(x,y)\in U$, then we have:
\begin{center}
		$d(p,\mathbb{R}^{2}\setminus U)<max(d(p,\mathbb{R}^{2}\setminus U_{1}),d(p,\mathbb{R}^{2}\setminus U_{2}))$.
	\end{center}
Now take the definable cover  $(U'_{1},U'_{2})$ defined by:
\begin{center}
$U'_{1}=\{(x,y)\in U \; ; \; y<x^{2}  \}$\\
$U'_{2}=\{(x,y) \in  U \; ; \; y>-x^{2}   \}$
\end{center}
Then $(U'_{1},U'_{2})$ is not a regular cover. Assume that this is a regular cover, take a point $p_{a}=(a,0)\in U$. Then we have:
\begin{center}
 $d(p_{a},\mathbb{R}^{2}\setminus U)=a$ for $a<\frac{1}{2}$ \\
 $d(p_{a},\mathbb{R}^{2}\setminus U'_{1})=d(p_{a},\mathbb{R}^{2}\setminus U'_{2})<a^{2} $.
 \end{center} 
Hence we have for $a<\frac{1}{2}$:
\[
\frac{d(p_{a},\mathbb{R}^{2}\setminus U)}{max(d(p_{a},\mathbb{R}^{2}\setminus U'_{1}),d(p_{a},\mathbb{R}^{2}\setminus U'_{2}))}=\frac{d(p_{a},\mathbb{R}^{2}\setminus U)}{d(p_{a},\mathbb{R}^{2}\setminus U'_{1})}>\frac{a}{a^{2}}=\frac{1}{a}\mapsto _{a\mapsto 0} +\infty .
\]
Hence this is a contradiction with the definition of a regular cover.\\

\begin{thm}\label{thm3}
For any open relatively compact definable subset $U$ of $\mathbb{R}^{n}$ there exists a regular cover.
\end{thm}
\begin{proof}
Take $X=\partial U=\overline{U}\setminus U$, $X$ is compact and definable of dimension $n-1$. Take $Z=Sing(X)$ the set of the points where $X$ is not $n-1$ dimensional submanifold of $\mathbb{R}^{n}$ near this points (i.e that $Z=(Reg(X))^{c}$ ). For any linear projection $P:\mathbb{R}^{n}\longrightarrow \mathbb{R}^{n-1}$ (i.e, a projection with respect to a vector in $\mathbb{R}^{n} $), we define the discriminant $\Delta _P$ by:
\begin{center}
$\Delta _P=P(Z) \cup CV(P _{\mid Reg(X)})$,
\end{center}
where $CV(\pi _{\mid Reg(X)})$ is the set of critical values of $P$ on $Reg(X)$.\\
It's obvious that $\Delta _P$ is definable subset of $\mathbb{R}^{n-1}$, because $P$ is a definable map and $CV(P _{\mid Reg(X)})$ can be described by a first order formula. And we have also :
\begin{center}
$\overline{P(U)}=P(U)\cup \Delta _P$.
\end{center}
Take $\Lambda =\{\pi _{1},...,\pi _{k}  \}$ a set of $\varepsilon$-weak regular projection with respect to $X$, and $v_1 ,...,v_k \in \mathbb{R}^{n-1}$ such that $\pi _j = \pi _ {v_j}$. For $x\in \mathbb{R}^{n}$ we denote by $C_{\varepsilon}^j (x)$ the cone:
\begin{center}
$C_{\varepsilon}^j (x)=C_{\varepsilon } (x,v_j)= \{   x+t(v,1) \; ; \; t\in \mathbb{R}^\ast, \; v\in B(\varepsilon , v_j)    \}$.
\end{center}
\begin{lem} \label{lemma10}
Take $\pi _j \in \Lambda$, and define:
\begin{center}
$R(\pi _j)=\{ x \in U \; : \;  \pi _j \; is \; \varepsilon -weak \; regular \; at \; x \; with \; respect \; to \; X \} $.
\end{center}
Then we have $\pi _j (R(\pi _j )) \cap \Delta _{\pi _j }=\emptyset$ and there is some $C > 0$ such that for all $x\in R(\pi _j)$ we have:
\begin{center}
$(3.1) \; \; \; \; \; \; \; \; \; \; d(x,X\setminus C_{\varepsilon}^j (x))\leq Cd(\pi _j (x),\pi _j (X \setminus C_{\varepsilon} ^j (x))) \leq C d(\pi _j (x),\Delta _{\pi _j })$.
\end{center}
\end{lem}
\begin{proof}
It's enough to assume that $\pi _j$ is the standard projection $\pi _j=\pi _0=\pi$. If $\pi  (R(\pi )) \cap \Delta _{\pi   } \neq \emptyset$,  then there is $x'\in \pi  (Z) \cup CV(\pi  _{\mid Reg(X)})$ such that $x'=\pi (x)$ and $x\in R(\pi )$, but since $\pi $ is $\varepsilon$-weak regular a $x$  then $\pi   ^{-1} \pi (x) \subset (Reg(X)) $, hence the contradiction, so $\pi (R(\pi )) \cap \Delta _{\pi }=\emptyset$.\\
Since $\Delta _{\pi } \subset \pi (X\setminus C_{\varepsilon} ^0(x))$, we have $d(\pi (x),\pi (X \setminus C_{\varepsilon} ^0(x))) \leq  d(\pi (x),\Delta _{\pi })$, and to show the second claim it's enough to find $C> 0$ such that:

\begin{center}
 
 $d(x,X\setminus C_{\varepsilon} ^0(x))\leq Cd(\pi  (x),\pi (X \setminus C_{\varepsilon} ^0(x)))$.
\end{center} 
But for any $x'\notin C_{\varepsilon} ^0(x)$ we have:

\begin{center}
$\frac{d(x,x')}{d(\pi (x),\pi (x'))}\leqslant 1+ \frac{1}{\varepsilon}$
\end{center}
Indeed, take $x'=x+t(v,1)\notin C_{\varepsilon } ^0(x)$, with $t\in \mathbb{R}^{\ast}$ and $\parallel v  \parallel \geqslant \varepsilon$, hence $\pi  (x')=\pi (x)+tv$, then we have:

\begin{center}
$\frac{d(x,x')}{d(\pi (x),\pi (x'))}=\frac{\mid t \mid \parallel (v,1) \parallel}{\mid t \mid \parallel v \parallel} \leqslant 1+\frac{1}{\varepsilon}$.
\end{center}
Now for any $x'\in X\setminus C_{\varepsilon}(x)$ we have: 

\begin{center}
$d(x,x')\leqslant (1+ \frac{1}{\varepsilon})d(\pi (x),\pi (x'))$.
\end{center}
But since this is for any $x'\in X\setminus C_{\varepsilon}(x)$, by the definition of the infimum we deduce that:  

\begin{center}

$d(x,X\setminus C_{\varepsilon}(x))\leq (1+\frac{1}{\varepsilon})d(\pi (x),\pi (X \setminus C_{\varepsilon}(x)))$.

\end{center}
\end{proof}

\begin{rem} For the left side inequality in Lemma ~~\ref{lemma10}, we don't need $x$ to be in $R(\pi _j)$. 
\end{rem}

For the proof of Theorem ~\ref{thm3} we proceed by induction on $n$. Assume that Theorem ~\ref{thm3} is true in $\mathbb{R}^{n-1}$. Fix $j\in \{ 1,...,k \}$, then by the induction assumption and by Lemma~\ref{lemma10} there is a $C_{j}>1$ and a finite definable cover $(U_{j,i})_{i\in I_j  }$ of $\pi _{j}(U)\setminus \Delta _{\pi _j}$ such that we have for all $x'\in \mathbb{R}^{n-1}$:

\[
d(x',\mathbb{R}^{n-1}\setminus (\pi _{j}(U)\setminus \Delta _{\pi _j}))\leq C_{j}max _{i}d(x', \mathbb{R}^{n-1}\setminus(U_{j,i}))
\]

and for all $x\in R(\pi _j)$ we have by Lemma~\ref{lemma10}: 
\[
d(x,X\setminus C_{\varepsilon}^j (x))\leq C_j d(\pi _j (x),\pi _j (X \setminus C_{\varepsilon} ^j (x))) \leq C_j d(\pi _j (x),\Delta _{\pi _j }).
\]

Now take a cell decomposition of $\mathbb{R}^{n}$ compatible with $U$, $X$, $Z$, $CP((\pi _{j})_{\mid _{X}})$, and $\pi _j (U)\setminus \Delta _{\pi _j}$. Then for each $i\in I_{j}$ there are a definable functions:

\begin{center}
$\phi _{1}<\phi _{2}<...<\phi _{l_i} :U_{j,i}\longrightarrow \mathbb{R}$
\end{center}
such that $X\cap \pi _{j}^{-1}(U_{j,i})$ is the disjoint union of graphs of these functions, and $U\cap \pi _{j}^{-1}(U_{j,i})$ is the disjoint union of the open sets bounded by the graphs of these functions. So for every $j\in \{1,...,k  \}$ and $i\in I_j$ we have that : 

\begin{center}
$U\cap \pi _{j}^{-1}(U_{i,j})=\bigcup _{m\in M_{j,i}} U_{j,i,m}$ ,
\end{center}

where $U_{j,i,m}$ are open definable subsets of $U$,  given by:

\begin{center}
$U_{j,i,m}=\{(x',x_{n}) \; : \; x'\in U_{j,i} \; and         \; \phi _{m_1}(x')<x_n <\phi _{m_2}(x')    \; \}$, with $\Gamma _{ \phi _{m_1} }\subset X$ and $\Gamma _{ \phi _{m_2} }\subset X$.
\end{center}

Now take $x\in U$, hence by the weak projection theorem there is a projection $\pi _j\in \Lambda$ such that $x\in R(\pi _j)$. Let's consider $i\in I_j$ such that $\pi _j (x) \in U_{j,i}$ and : 
\begin{center}
$(3.2)\; \; \; \; \; \; \; \; \; \; \; d(\pi _j (x),\mathbb{R}^{n-1}\setminus (\pi _{j}(U)\setminus \Delta _{\pi _j}))\leq C_{j}d(\pi _j (x), \mathbb{R}^{n-1}\setminus(U_{j,i}))$.
\end{center}

Since $\partial (\pi _{j}(U)\setminus \Delta _{\pi _j}) \subset \Delta _{\pi _j}$, we have :

\begin{center}
$(3.3)\; \; \; \; \; \; \; \; \; \; \;d(\pi _j (x), \Delta _{\pi _j})\leq C_{j}d(\pi _j (x), \mathbb{R}^{n-1}\setminus(U_{j,i}))$.
\end{center}

Now take $m\in M_{j,i}$ such that $x\in U_{j,i,m}$, then we claim that :

\begin{center}
$(3.4)\; \; \; \; \; \; \; \; \; \; \;d(x,X)\leqslant (C_j)^{2}d(x,\mathbb{R}^{n}\setminus U_{j,i,m})$.
\end{center}

To prove $(3.4)$ we discuss two cases ( the first is obvious ) :

\begin{itemize}
\item[$(1)$] $d(x,\mathbb{R}^{n}\setminus U_{j,i,m})\geqslant d(x,X)$.
\item[$(1)$] $d(x,\mathbb{R}^{n}\setminus U_{j,i,m}) < d(x,X)$. In this case let $V=\partial U_{j,i,m} \cap \pi _j ^{-1} (\partial U_{j,i})$ ( the vertical part of $\partial U_{j,i,m}$ ). We have then:  
  
\begin{center}
  $d(x,\mathbb{R}^{n}\setminus U_{j,i,m})=d(x,V)$,
  \end{center}  
because $d(x,\mathbb{R}^{n}\setminus  U _{j,i,m})=min\{\; d(x,V), \; d(x,\partial U_{j,i,m} \cap X)     \}$. Therefore by $(3.1)$ and $(3.3)$ we have :

\begin{center}
	\begin{eqnarray*}
	d(x,X)&\leqslant & d(x,X\setminus C_{\varepsilon}^j (x)) \\
		&\leqslant & C_j d(\pi _j (x), \Delta _{\pi _j}) \\
		&\leqslant & C_j ^{2} d(\pi _j (x), \mathbb{R}^{n-1} \setminus U_{j,i}) \\
		&\leqslant & C_j ^{2} d(x,V) \\
		&\leqslant & C_j ^{2} d(x,\mathbb{R}^{n} \setminus U_{j,i,m}).
	\end{eqnarray*}
\end{center}  
And this proves $(3.4)$.\\
Finally, we have a finite cover $(U_{j,i,m})_{m\in M_{j,i},i\in I_{j},j}$ of $U$. Take $C=max_{j}C_{j}^{2}$. Hence for $x\in U $ we have:
\[
d(x,\mathbb{R}^{n}\setminus U)=d(x,X)\leqslant (C_{j})^{2}d(x,\mathbb{R}^{n}\setminus U_{j,i,m})\leqslant Cmax_{j,i,m}d(x,\mathbb{R}^{n}\setminus U_{j,i,m}).
\]
  
\end{itemize}

\end{proof}

\textbf{Acknowledgment.}  The author is very grateful to his thesis advisor Adam Parusi\'nski for his help and support, and the long hours of discussion he
devoted to him during the preparation of this work.

\end{document}